\PassOptionsToPackage{dvipsnames}{xcolor}
\documentclass[10pt]{amsart}
\usepackage[margin=3cm]{geometry}

\usepackage[utf8]{inputenc}  
\usepackage{graphicx}
\usepackage{pgfplots}
\usepackage{amsmath,amssymb, mathtools}
\usepackage{multicol}
\usepackage{hyperref}
\usepackage{cleveref}
\usepackage{dsfont}
\usepackage{xcolor}
\usepackage{tikz}
\usetikzlibrary{intersections,positioning}

\newtheorem{theorem}{Theorem}[section]
\newtheorem{proposition}{Proposition}[section]
\newtheorem{lemma}{Lemma}[section]

\newtheorem{remark}{Remark}[section]
\newtheorem*{maintheorem*}{Main Theorem}
\allowdisplaybreaks
\numberwithin{equation}{section}

\usepackage{todonotes}

\newcommand{\N}{\mathbb{N}}

\newcommand{\R}{\mathbb{R}}

\newcommand{\pt}{\partial_t}

\newcommand{\ptt}{\partial_{t}^2}

\newcommand{\eps}{\varepsilon}

\DeclareMathOperator*{\supp}{supp}

\newcommand{\dd}{\,\mathrm{d}}

\renewcommand{\S}{\mathbb S}
\makeatletter
\@namedef{subjclassname@2020}{%
	\textup{2020} Mathematics Subject Classification}
\makeatother

\begin{document}

\title[Stationary critical points of the fractional heat flow]{The stationary critical points of the fractional heat flow}

\subjclass[2020]{35R11, 35S15, 47G20, 35A01, 35B05, 35K05, 35K15.}
\keywords{Fractional Laplacian, fractional heat equation, fractional heat kernel, spatial critical points, hot spots. \vspace{1.5mm}}

\author[N. De Nitti]{Nicola De Nitti}
\address[N. De Nitti]{Friedrich-Alexander-Universit\"{a}t Erlangen-N\"{u}rnberg, Department of Data Science, Chair for Dynamics, Control and Numerics (Alexander von Humboldt Professorship), Cauerstr. 11, 91058 Erlangen, Germany.}
\email{nicola.de.nitti@fau.de}
\author[S. Sakaguchi]{Shigeru Sakaguchi}
\address[S. Sakaguchi]{Graduate School of Information Sciences, Tohoku University, Sendai 980-8579, Japan.}
\email{sigersak@tohoku.ac.jp}


\begin{abstract}
We study the spatial critical points of the solutions $u=u(x,t)$ of the fractional heat equation. For the Cauchy problem, we show that the origin $0$ satisfies $\nabla_x u(0,t) = 0$ for $t>0$ if and only if the initial data satisfy a balance law of the form  $\int_{\mathbb{S}^{N-1}} \omega u_0(r\omega) \dd \omega=0$ for a. e. $r \ge 0$. Moreover, for the Dirichlet initial-boundary value problem, we prove two symmetry results: $\Omega$ is a ball centered at the origin if and only if $\nabla_x u(0,t) = 0$ for $t>0$ provided that the initial data satisfies the above mentioned balance law; $\Omega$ is centrosymmetric if and only if $\nabla_x u(0,t) = 0$ for $t>0$ provided that the initial data is centrosymmetric. These results extend some theorems obtained by Magnanini and Sakaguchi in 1997-1999 for the (local) heat equation to the fractional context. These extensions are nontrivial, because of the nonlocal nature of the fractional Laplacian. Among others, the proof of the characterization of a ball in the Dirichlet initial-boundary value problem for the fractional heat flow not only works for the classical heat flow but also gives a new insight into the problem. 
\end{abstract}

\maketitle


\section{Introduction}
\label{sec:intro}

We study the spatial critical points of solutions of the fractional heat equation  
\begin{align}\label{eq:fhe}
\begin{cases}
\pt u(x,t) + (-\Delta)^s u(x,t) = 0, & x \in \R^N, \ t >0, \\ 
u(x,0) = u_0(x), & x \in \R^N,
\end{cases}
\end{align}
where $u_0 \in L^\infty(\R^N)$,  $0 < s < 1$, and the operator $(-\Delta)^s$ is the \emph{fractional Laplacian} (see  \cite{BuVa2016,MR3967804} for further information), that is, 
\begin{equation}
\label{frac-lap-sg-int}
(-\Delta)^s u(x) := c_{N,s} \lim_{\varepsilon\to 0^+}\int_{|y-x|>\varepsilon} \frac{u(x) - u(y)}{|x-y|^{N+2s}} \dd  y,
\end{equation}
where $c_{N,s}$ is a constant chosen so that the Fourier representation 
\[ (-\Delta)^s u = \mathcal F^{-1} (|\xi|^{2s} \mathcal F u) \]
holds in $\R^N$,  namely, 
\begin{align}\label{eq:constant}
c_{N,s} :=   \frac{2^{2s}s \Gamma(\frac{N+2s}{2})}{\pi^{N/2}  \Gamma (1-s)}. 
\end{align} 

For the classical heat equation, corresponding to the case $s = 1$ of \eqref{eq:fhe}, the study of spatial critical points goes back to \cite{MR1094719}, where hot spots (i.e. critical points where the solutions attains its spatial maximum) have been considered: the author studied the location of the hot spots of the non-negative solution of the Cauchy problem and their asymptotic behavior as time goes to infinity. We refer to \cite{MR1272155} for the initial-boundary value problem on unbounded domains in $\R^N$. 

On this topic, a conjecture formulated by Klamkin (and modified by Kawohl)   states that the hot spots of the initial-boundary value problem on bounded convex domains in $\R^N$ do not move in time for positive constant initial data under the homogeneous Dirichlet boundary condition, then the convex domain must have some sort of symmetry. Some progress on this conjecture has been made in  \cite{klamkin}. This problem was then studied further  in \cite{MR1454253},  where the authors proved that a solution $u$ has a spatial
critical point not moving along the heat flow if and only if $u$ satisfies some balance law (which amounts to a symmetry condition with respect to the critical point). Moreover, they provided a characterization of balls by making use of the spatial critical points not moving along the heat flow. A similar characterization was given in \cite{MR1686571,MR1718642} for centrosymmetric domains. In a similar spirit, in \cite{MR1979777}, planar domains which are invariant by some rotation have been studied.  In \cite{MR1810137}, the stationary critical points of the heat flow have also been studied in the sphere and hyperbolic spaces. We also refer to \cite{MR1929895} for a survey of this research area. 

In the present paper, we extend some results of \cite{MR1454253,MR1686571} to the fractional context. These extensions are nontrivial, because of the nonlocal nature of the fractional Laplacian. First, we consider the Cauchy problem \eqref{eq:fhe}  and prove that the origin is a critical point of $u$ which does not move along the fractional heat flow if and only if the initial datum satisfies a balance law of the form 
\begin{align}\label{eq:balance-law}
\int_{\mathbb{S}^{N-1}} \omega u_0(r\omega) \dd \omega = 0 \ \text{ for any $r \ge 0$,}
\end{align}
where $\mathbb{S}^{N-1}$ denotes the unit sphere centered at the origin in $\mathbb R^N$ with the volume element $\dd\omega$, is satisfied (see \cite[Theorem 1]{MR1454253} for $s=1$). For the special case $s=1/2$, where the fractional heat kernel can be computed explicitly (see Section \ref{sec:preliminaries}), such result was recently obtained in \cite[Theorem 3.7]{MR4036719}.

Next, we consider the initial-boundary value problem with homogeneous Dirichlet data 
\begin{align}\label{eq:fhe-bc}
\begin{cases}
\pt u(x,t) + (-\Delta)^s u(x,t) = 0, & x \in \Omega, \ t >0, \\ 
u(x,0) = u_0(x), & x \in \Omega,\\
u(x,t) = 0, & x \in \R^N \setminus \Omega, \ t >0
\end{cases}
\end{align}
for a bounded $C^{1,1}$ domain $\Omega \subset \R^N$ with $0\in\Omega$. We prove the following symmetry results: 
\begin{itemize}
	\item if $u_0$ satisfies the balance law of the type \eqref{eq:balance-law}, then $\nabla_x u(0,t) = 0$ for $t >0$ and every $u_0$ is equivalent to $\Omega$ being a ball centered at the origin (cf. \cite[Theorem 4]{MR1454253});
	\item if $u_0$ satisfies $u_0(x) = u_0(-x)$, then $\nabla_x u(0,t) = 0$ for $t >0$ and every $u_0$  is equivalent to $\Omega$ being centrosymmetric with respect to the origin (cf. \cite[Theorems 1 \& 2]{MR1686571}),
\end{itemize}
where  the proof of the above characterization of a ball  for the fractional heat flow not only works for the classical heat flow but also gives a new insight into the problem (see \textbf{Step 3:} {\it A construction to represent} $\nabla_xG(0,y)$ of the proof of Theorem \ref{thm:sym}  \textbf{((1)$\implies$(2))} in Section \ref{ssec:radial_sym}).

For the regularity (up to the boundary) of solutions to the fractional heat equation, we refer to \cite{MR3462074}, where problem \eqref{eq:fhe-bc} is solved by the method of separation of variables with the aid of the eigenfunctions of the fractional Laplacian.

Finally, as in \cite[Sections 4 \& 5]{MR1686571}, we note that these results can be extended to spatial zero points (that is, if,  instead of $\nabla_x u(\cdot, t)=0$, we consider $u(\cdot, t)=0$ for each time $t\ge 0$) and that all the results on the fractional heat equation can be restated for (smooth solutions of) the fractional wave equation 
\begin{align}\label{eq:fwCauchy}
\begin{cases}
\ptt w(x,t) + (-\Delta)^s w(x,t) = 0, & x \in \R^N, \ t >0, \\ 
w(x,0) = 0 \mbox{ and }  \pt w(x,0)=u_0(x), & x \in \R^N,
\end{cases}
\end{align}
where $u_0 \in C^\infty_0(\R^N),$ and
\begin{align}\label{eq:fw}
\begin{cases}
\ptt w(x,t) + (-\Delta)^s w(x,t) = 0, & x \in \Omega, \ t >0, \\  w(x,0) = 0, & x \in \Omega,\\
\pt w(x,0) = u_0(x), & x \in \Omega,\\
w(x,t) = 0, & x \in \R^N \setminus \Omega, \ t >0,
\end{cases}
\end{align}
where $u_0 \in C^\infty_0(\Omega)$ (for which see \cite{MR4221596,MR4124319,MR4097647,MR3987171,MR3903611,2105.11324}) thanks to the properties of the Laplace transform. We note that problem \eqref{eq:fw} can also be solved by the method of separation of variables as problem  \eqref{eq:fhe-bc} is solved in \cite{MR3462074}.

\subsection{Outline}
\label{ssec:outline}
The paper is organized as follows. In Section \ref{sec:preliminaries}, we present some preliminary notions on fractional Sobolev spaces, the fractional Laplace operator,  the Green's function and the fractional heat kernel that are needed throughout the paper. In Section \ref{sec:main}, we state our main results outlined above. In Section \ref{sec:proofs}, we present the proofs. First, we deal with the critical points not evolving along the heat flow: in Section \ref{ssec:cauchy}, we consider the Cauchy problem; in Section \ref{ssec:radial_sym}, we consider the radial symmetry result; and, in Section \ref{ssec:centro-sym}, the central symmetry result. Then, in Sections \ref{sec:proofs-z} and \ref{sec:proofs-w}, we outline how the  previously developed arguments can be adapted to deal with  spatial zero points of solutions of the heat equation and also with critical and zero points of the fractional wave equation.

\section{Preliminaries and notation}
\label{sec:preliminaries}

\subsection{Fractional Sobolev spaces}

Following mostly the notations in \cite{MR3916700}, we define
\[ \mathcal E_s[u,v] = \iint_{\R^{2N}} \frac{(u(x)-u(y))(v(x)-v(y))}{|x-y|^{N+2s}} \dd x \dd y \]
and abbreviate $\mathcal E_s[u] = \mathcal E_s[u, u]$. We denote by 
\[ H^s(\R^N) := \{ u \in L^2(\R^N)\, :\,  \mathcal E_s[u] < \infty \} \]
the (non-homogeneous) Sobolev space of order $s$. 
For a bounded smooth open set $\Omega \subset \R^N$, 
we define the function space 
\[ \widetilde{H}^s(\Omega) := \{ u \in H^s(\R^N) \, : \, u \equiv 0  \text{  on  } \R^N \setminus \Omega \}. \]
The quadratic form $\mathcal E_s$ is closed on $\widetilde{H}^s(\Omega)$ and it is therefore the quadratic form of a self-adjoint operator on $L^2(\Omega)$, which is indeed the fractional Laplacian as
\[ \mathcal E_s[u] = 2 C_{N,s}^{-1} (u, (-\Delta)^s u)_{L^2(\R^N)} = 2 C_{N,s}^{-1} \| (-\Delta)^{s/2} u\|_2^2. \] Denoting by $H^s_0(\Omega)$ the closure of $C^\infty_0(\Omega)$ with respect to $\mathcal E_s$, we have $\widetilde{H}^s(\Omega) \subset H^s_0(\Omega)$ with equality if and only if $s - \frac{1}{2} \notin \mathbb Z$, see \cite[Proposition 1]{MuNa2016} for a more detailed statement. 
We remark that the quantity 
\begin{align}\label{eq:fr_norm}
\Vert u \Vert_{L^2(\R^N)} + \left(\int_{\R^N \times \R^N} \frac{ |u(x) - u(y)|^2}{|x-y|^{N+2s}} \dd x \dd y \right)^{1/2} < \infty
\end{align} 
defines a norm on both $H^{s}(\R^N)$ and $\tilde{H}^s(\Omega)$ and that they are Hilbert spaces with the inner product 
 given by 
\begin{align}\label{eq:inner}
\langle u,v \rangle_{L^2(\R^N)} + \mathcal E_s[u,v].
\end{align}

\subsection{Fractional heat equation}
\label{ssec:fractional-he}

For $0 < s \le 1$, we define the \emph{fractional heat kernel} $P(x,t;s)$ as the solution of the problem 
\begin{align}\label{eq:fhe-d}
\begin{cases}
\pt P(x,t;s) + (-\Delta)^s P(x,t;s) = 0, & x \in \R^N, \ t >0, \\
P(x,0;s) = \delta_{0}, & x \in \R^N.
\end{cases}
\end{align}
Using the Fourier transform, we obtain that 
\begin{align}\label{eq:k-fourier}
P(x,t;s) = \mathcal F^{-1}(e^{-t(2\pi|\xi|)^{2s}}),
\end{align}
from which we also deduce $P(\cdot,t;s) \in C^\infty(\R^N)$ (see \cite[Section 8.3]{MR3614666} for finer Gervey regularity properties, depending on the value of $0 < s \le 1$). 

The fractional heat kernel $P(x,t; s)$ can be computed explicitly for $s = 1$ (the classical heat kernel) and for $s = 1/2$:
\begin{align}
\label{eq:k1} P(x,t;1) &= (4\pi t)^{-N/2} e^{-|x|^2/4t}, \\
\label{eq:k12} P(x,t;1/2) &= \frac{c_N t}{(t^2+|x|^2)^{(N+1)/2}},
\end{align}
where 
\begin{align}\label{eq:cn}
c_N := \frac{\Gamma(N+1/2)}{\pi^{N+1/2}}.
\end{align}

We remark that \eqref{eq:k12} is the Poisson kernel for the Laplace equation posed in the upper half-space
$$H_{N}^+ = \{(x, t): x \in  \R^N, \ t >0\} \subset \R^{N+1},$$
which is consistent with the Caffarelli-Silvestre extension theorem (see \cite{MR2354493}). 

For general values $s \in (0,1)$, the heat kernel has the self-similar form 
\begin{align}\label{eq:ks}
	P(x,t;s) = t^{-N/2s}\Phi_s(|x|t^{-1/2s})
\end{align}
but its profile is not simple  (see \cite[Eqs. (16.11) and (16.14), p. 97]{MR3916700}).
However, the following estimate holds (see \cite[Theorem 16.6]{MR3916700}):
\begin{align}\label{eq:estks}
P(x,t;s) \asymp \frac{t}{(t^{1/s} + |x|^2)^{(N+2s)/2}}.
\end{align}
We refer to \cite[Section 16]{MR3916700} and \cite{BuVa2016} for further details and references on the fractional heat semigroup. 

By using the fractional heat kernel $P(x,t;s)$, the fractional heat equation can be solved for all $0 < s \le 1$: 
\begin{align}\label{eq:sol}
u(x,t) = \int_{\R^N} P(x-y,t;s) u_0(y) \dd y.
\end{align}

The optimal class of initial data for which this representation holds is given in \cite{MR3614666}: it is the class of locally finite Radon measures satisfying the growth condition
\begin{align}\label{eq:growth}
\int_{\R^N} (1+|x|^2)^{-(N+2s)/2} \dd |\mu|(x) < \infty.
\end{align}

\subsection{Fractional Poisson equation}
\label{ssec:fractional-poisson}

For the fractional Laplacian, we observe that the fundamental solution  (see \cite[Theorem 2.3]{MR3461641}), i.e. the solution of  
\begin{align}\label{eq:fl-fs}
(-\Delta)^s \Psi =\delta_0, \quad x \in \R^N,
\end{align}
is given by 
\begin{align}\label{eq:fund}
\Psi(x) = \begin{cases}
\frac{\Gamma(N/2-s)}{2^{2s}\pi^{N/2}\Gamma(s)} |x|^{-N+2s} & \text{ if } N \neq 2s, \\
-\frac{1}{\pi}\log|x| & \text{ if } N = 2s.
\end{cases}
\end{align}
Moreover, the Green's function of the fractional Laplacian on a ball (see \cite[Theorem 3.1]{MR3461641}) is explicitly given by 
\begin{align}\label{eq:green1}
G(x,y) := \kappa(N,s) |x-y|^{2s-N}\int_0^{r_0(x,y)} \frac{t^{s-1}}{(t+1)^{N/2}} \dd t
\end{align}
with 
\begin{align*}
r_0(x,y) = \frac{(r^2-|x|^2)(r^2-|y|^2)}{r^2|x-y|^2}
\end{align*}	
in case $N \neq 2s$ or by 
\begin{align}\label{eq:green2}
G(x,y) = \kappa(1,1/2) \log\left(\frac{r^2-xy+\sqrt{(r^2-x^2)(r^2-y^2)}}{r|y-x|} \right)
\end{align}
in case $N=2s$, where 
\begin{align*}
\kappa(N,s) &= \frac{\Gamma(N/2)}{2^{2s}\pi^{N/2}\Gamma^2(s)} && \text{ if } N \neq 2s, \\
\kappa(1,1/2) &= \frac{1}{\pi} && \text{ if } N = 2s.
\end{align*}
For a general domain $\Omega$, the properties of the Green's function have also been studied; we particularly recall the following result (see \cite[Proposition 2.5]{MR1490808} and \cite[Eq. (1.65)]{MR2569321}).
\begin{proposition}[Green's function of the fractional Laplacian]\label{lm:prob-formula}
	Let $\Omega$ be an open set in $\R^N$, $B_r(x_1)\subset \Omega$, $x \in B_r(x_1)$, and $y \in \R^N$. Then 
	\begin{align*}
	G(x,y) \ge \int_{\R^N\setminus\Omega} G(u,y)P_r(x-x_1,u-x_1) \dd u 
	\end{align*}
	and, if $y \notin \overline{B_r(x_1)}$, 
	\begin{align}
	\label{Cauchy like formula}
	G(x,y) = \int_{\R^N \setminus B_r(x_1)} G(u,y) P_r(x-x_1,u-x_1) \dd u,
	\end{align}
	where 
	\begin{align}
P_r(x,y) := \begin{cases}
\Gamma\left(\tfrac{N}{2}\right)\pi^{-\tfrac{N}{2}-1}\sin(\pi s) \left(\frac{r^2-|x|^2}{|y|^2-r^2} \right)^{s} |x-y|^{-N} & \text{ for } |y|>r, \\
0 & \text{ for } |y|\le r.
\end{cases}
	\end{align}
\end{proposition}

In particular, formula \eqref{Cauchy like formula}, together with the symmetry of $G(x,y),$ guarantees that $G(x,y)$ is real analytic in $\{x\not= y\}.$

We remark that \cite[Theorem 3.2]{MR3461641} gives that, for $r>0$, $h \in C^{2s + \eps}(B_r(0)) \cap C(\bar B_r)$, 
the function 
\begin{align*}
u(x) := \begin{cases}
\int_{B_r(0)} h(y) G(x,y) \dd y & \text{ if } x \in B_r(0), \\
0 & \text{ if } x \in \R^N \setminus B_r(0)
\end{cases}
\end{align*}
is the unique pointwise continuous solution of the problem 
\begin{align*}
\begin{cases}
(-\Delta)^s u(x) = h(x), & x \in B_r(0), \\
u(x)= 0, & x \in \R^N \setminus B_r(0). 
\end{cases}
\end{align*}

Furthermore, we recall that a function $u$ is $s$-harmonic in $x \in \R^N$ if and only if it satisfies the following $s$-mean value property (see  \cite[Definition 2.1]{MR3461641} or  \cite[Section 15]{MR3916700}): 
\begin{align}\label{eq:mean-value}
u(x) = \int_{\R^N \setminus B_r(0)} A_r(y)u(x-y) \dd y,
\end{align}
where 
\begin{align*}
A_r(y) := \begin{cases}
c(N,s) \frac{r^{2s}}{(|y|^2-r^2)|y|^N}, & y \in \R^N \setminus \bar B_r(0), \\ 
0, & y \in \bar B_r(0). 
\end{cases}
\end{align*}

For a general bounded domain $\Omega$, we consider 
\begin{align}\label{eq:fl-h1}
\begin{cases}
(-\Delta)^s u(x) = h(x), & x \in \Omega, \\
u(x)= 0, & x \in \R^N \setminus \Omega.
\end{cases}
\end{align}
We can then prove that $\| u\|_{\widetilde{H}^s(\Omega)} \le C_0 \|h\|_{\widetilde{H}^s(\Omega)}$ using Poincar\'e's inequality. 

\begin{proposition}\label{prop:estimate1}
	There exists a constant $C_0>0$ such that the solution $u$ of \eqref{eq:fl-h1} satisfies 
	$$\| u\|_{\widetilde{H}^s(\Omega)} \le C_0 \|h\|_{\widetilde{H}^s(\Omega)}.$$ 
\end{proposition}

\begin{proof}
	By the Poincar\'e-type inequality in \cite[Section 2.3]{MR3318251}, we have that there exists $C_P>0$ such that 
	\begin{align*}
	\|u\|_{L^2(\R^N)}^2 \le C_P \mathcal E_s[u]
	\end{align*}
	By the weak formulation of \eqref{eq:fl-h1}
	\begin{align}\label{eq:weak-E}
	\mathcal E_s[u,v] = \langle h, v\rangle_{L^2(\R^N)}
	\end{align}
	for all $v \in \widetilde{H}^s(\Omega)$; setting $u=v$, we have 
	\begin{align*}
	\mathcal E_s[u] &= \langle h,u\rangle_{L^2(\R^N)} \\ &\le \|h\|_{L^2(\R^N)}\|u\|_{L^2(\R^N)} \le C_P\|h\|^2_{L^2(\R^N)},
	\end{align*}
	where we used Schwarz and Poincar\'e's inequalities.
	This yields that there exists $C_0>0$ such that 
	\begin{align*}
	\|u\|_{\widetilde{H}^s(\Omega)} \le C_0 \|h\|_{\widetilde{H}^s(\Omega)}.
	\end{align*}
\end{proof}

\begin{remark}\label{remark meaning of Green operator}
	Proposition \ref{prop:estimate1} means that the Green's function $G$ produces the Green's operator $G:\widetilde{H}^s(\Omega) \to \widetilde{H}^S(\Omega)$ which is a bounded operator with norm $\| G\| \le C_0$. 
\end{remark}

Similarly, for the problem 
\begin{align}\label{eq:fl-h2}
\begin{cases}
(-\Delta)^s u(x) + \lambda u(x) = h(x), & x \in \Omega, \\
u(x)= 0, & x \in \R^N \setminus \Omega,
\end{cases}
\end{align}
where $h \in \widetilde{H}^s(\Omega)$, we can find a solution $u \in \widetilde{H}^s(\Omega)$ and prove the following result. 

\begin{proposition}\label{prop:estimate2}
	For $\lambda >0$, the solution $u \in \widetilde{H}^s(\Omega)$ of \eqref{eq:fl-h2} satisfies 
	\begin{align*}
\|u\|_{\widetilde{H}^s(\Omega)} \le C_0 \|h\|_{\widetilde{H}^s(\Omega)},
\end{align*}
where $C_0$ is the same constant as in Proposition \ref{prop:estimate1}.
\end{proposition}

\begin{proof}
	The proof is similar to the one of Proposition \ref{prop:estimate1}. The only differences are as follows: we replace \eqref{eq:weak-E} by 
	\begin{align}\label{eq:weak-E2}
	\mathcal E_s[u,v]+ \lambda \langle u, v \rangle_{L^2(\R^N)}= \langle h, v\rangle_{L^2(\R^N)}
	\end{align}
	and observe that, setting $v=u$, yields 
	\begin{align*}
	\mathcal E_s[u] + \underbrace{\lambda \| u \|_{L^2(\R^N)}^2}_{\ge 0 \text{ as $\lambda >0$}} = \langle h,u\rangle_{L^2(\R^N)} &\le  \|h\|_{L^2(\R^N)}\|u\|_{L^2(\R^N)}\\ & \le C_P\|h\|^2_{H^s(\R^N)}.
	\end{align*}
\end{proof}

\section{Main results}
\label{sec:main}

\subsection{Critical points and the fractional heat flow}
\label{ssec:main}

Our first main result is the fractional version of \cite[Theorem 1]{MR1454253}, dealing with the Cauchy problem \eqref{eq:fhe}.

\begin{theorem}[Stationary critical points and a balance law for the Cauchy problem]\label{thm:1}
Let $s \in (0,1)$, $u_0 \in C^\infty_0(\mathbb R^N)$ with $\supp(u_0) \subset B_L(0)$ for some $L >0$, and $u$ be the solution of the Cauchy problem for the fractional heat equation \eqref{eq:fhe}. Then the following conditions are equivalent: 
\begin{enumerate}
	\item $\nabla_x u(0,t) = 0$ for any $t >0$; 
	\item $\displaystyle \int_{\mathbb{S}^{N-1}} \omega u_0(r\omega) \dd \omega = 0$ for any $r \ge 0$. 
\end{enumerate}
\end{theorem}

%

Motivated by this result, we consider the Dirichlet IBVP \eqref{eq:fhe-bc} and prove the following symmetry result (see \cite[Theorem 4]{MR1454253} for $s=1$ and more general Robin-type boundary conditions).

\begin{theorem}[Radial symmetry result]\label{thm:sym}
	Let $s \in (0,1)$ and $u$ be the solution of \eqref{eq:fhe-bc}.  Let $B_\delta(0)$ be a ball
	centered at the origin with radius $\delta > 0$ such that $B_\delta(0) \subset \Omega$. Assume that $\Omega$ is a $C^{1,1}$ bounded domain and that $u_0 \in C_0^\infty(\Omega)$,  $\supp u_0 \subset B_\delta(0)$, and that the balance law 
	\begin{align*}
	\int_{\mathbb{S}^{N-1}} \omega u_0(r\omega) \dd \omega = 0
	\end{align*}
	holds for any $r \in (0,\delta)$. 
	Then the following conditions are equivalent: 
	\begin{enumerate}
		\item $\nabla_x u(0,t) =0$ for any $t >0$ and any $u_0$ as above;
	\item  $\Omega = B_R(0)$ for some $R>0$. 
\end{enumerate}
\end{theorem}

Finally, we study a centrosymmetry result for the IBVP \eqref{eq:fhe-bc} and prove the fractional version of \cite[Theorems 1 \& 2]{MR1686571,MR1718642}. 

\begin{theorem}[Centrosymmetry result]\label{thm:centrosym}
	Let $s \in (0,1)$ and $u$ be the solution of \eqref{eq:fhe-bc}. Let $B_\delta(0)$ be a ball
	centered at the origin with radius $\delta > 0$ such that $B_\delta(0) \subset \Omega$.  Assume that $\Omega$ is a $C^{1,1}$ bounded star-shaped domain with respect to the origin and that $u_0 \in C_0^\infty(\Omega)$, $\supp u_0 \subset B_\delta(0)$, and $u_0(x) = u_0(-x)$ for $x \in B_\delta(0)$. Then, the following conditions are equivalent: 
	\begin{enumerate}
		\item $\nabla_x u(0,t) =0$ for any $t >0$ and any $u_0$ as above;
		\item $\Omega$ is centrosymmetric with respect to the origin (i.e. $x \in \Omega \implies-x \in \Omega$).
	\end{enumerate} 
\end{theorem}

%

\subsection{Spatial zero points and the fractional heat flow}
\label{ssec:main-z}
Let us consider the problems for spatial zero points. Namely, instead of $\nabla_x u(\cdot, t)=0$, we consider $u(\cdot, t)=0$ for each time $t\ge 0$. 
Then, along the similar arguments we can get all theorems by replacing the balance law $\int_{\mathbb{S}^{N-1}} \omega u_0(r\omega) \dd \omega = 0$ by $\int_{\mathbb{S}^{N-1}}  u_0(r\omega) \dd \omega = 0$.

\begin{theorem}[Stationary zero points and a balance law for the Cauchy problem]\label{thm:1-z}
	Let $s \in (0,1)$, $u_0 \in C^\infty_0(\mathbb R^N)$ with $\supp(u_0) \subset B_L(0)$ for some $L >0$, and $u$ be the solution of the Cauchy problem for the fractional heat equation \eqref{eq:fhe}. Then the following conditions are equivalent: 
	\begin{enumerate}
		\item $u(0,t) = 0$ for any $t >0$; 
		\item $\displaystyle \int_{\mathbb{S}^{N-1}}  u_0(r\omega) \dd \omega = 0$ for any $r \ge 0$. 
	\end{enumerate}
\end{theorem}

\begin{theorem}[Radial symmetry result]\label{thm:sym-z}
	Let $s \in (0,1)$ and $u$ be the solution of \eqref{eq:fhe-bc}.  Let $B_\delta(0)$ be a ball
	centered at the origin with radius $\delta > 0$ such that $B_\delta(0) \subset \Omega$. Assume that $\Omega$ is a $C^{1,1}$ bounded domain and that $u_0 \in C_0^\infty(\Omega)$,  $\supp u_0 \subset B_\delta(0)$, and that the balance law 
	\begin{align*}
	\int_{\mathbb{S}^{N-1}}  u_0(r\omega) \dd \omega = 0
	\end{align*}
	holds for any $r \in (0,\delta)$. 
	Then the following conditions are equivalent: 
	\begin{enumerate}
		\item $u(0,t) =0$ for any $t >0$ and any $u_0$ as above;
		\item  $\Omega = B_R(0)$ for some $R>0$. 
	\end{enumerate}
\end{theorem}

\begin{theorem}[Centrosymmetry result]\label{thm:centrosym-z}
	Let $s \in (0,1)$ and $u$ be the solution of \eqref{eq:fhe-bc}. Let $B_\delta(0)$ be a ball
	centered at the origin with radius $\delta > 0$ such that $B_\delta(0) \subset \Omega$.  Assume that $\Omega$ is a $C^{1,1}$ bounded star-shaped domain with respect to the origin and that $u_0 \in C_0^\infty(\Omega)$, $\supp u_0 \subset B_\delta(0)$, and $u_0(x) = - u_0(-x)$ for $x \in B_\delta(0)$. Then, the following conditions are equivalent: 
	\begin{enumerate}
		\item $u(0,t) =0$ for any $t >0$ and any $u_0$ as above;
		\item $\Omega$ is centrosymmetric with respect to the origin (i.e. $x \in \Omega \implies-x \in \Omega$).
	\end{enumerate} 
\end{theorem}

\subsection{Fractional wave equation and symmetry}
\label{ssec:main-w}

Finally, we remark that all the results above  hold true if, instead of the fractional heat equation, we consider the fractional wave equation.  

\begin{theorem}[Stationary critical points and a balance law for the Cauchy problem for the fractional wave equation]\label{thm:fw-1}
	Let $s \in (0,1)$, $u_0 \in C^\infty_0(\mathbb R^N)$ with $\supp(u_0) \subset B_L(0)$ for some $L >0$, and $w$ be the solution of the Cauchy problem for the fractional wave equation \eqref{eq:fwCauchy}. Then the following conditions are equivalent: 
	\begin{enumerate}
		\item $\nabla_x w(0,t) = 0$ for any $t >0$; 
		\item $\displaystyle \int_{\mathbb{S}^{N-1}} \omega u_0(r\omega) \dd \omega = 0$ for any $r \ge 0$. 
	\end{enumerate}
\end{theorem}

\begin{theorem}[Radial symmetry for the wave equation]\label{thm:fw-radial}
	Let $s \in (0,1)$ and $w$ be the solution of \eqref{eq:fw}. Let $B_\delta(0)$ be a ball centered at the origin with radius $\delta > 0$ such that $B_\delta(0) \subset \Omega$. Let us consider the Laplace transform $W_\lambda(x) = \int_0^\infty e^{-\lambda t}w(x,t) \dd t$ (for $\lambda >0$).   Assume that $\Omega$ is a $C^{1,1}$ bounded star-shaped domain with respect to the origin and that $u_0 \in C_0^\infty(\Omega)$, $\supp u_0 \subset B_\delta(0)$ and that the balance law 
	\begin{align*}
	\int_{\mathbb{S}^{N-1}} \omega u_0(r\omega) \dd \omega = 0
	\end{align*}
	holds for any $r \in (0,\delta)$. 
	Then the following conditions are equivalent: 
	\begin{enumerate}
		\item $\nabla_x W_\lambda(0) =0$ for any $\lambda >0$ and any $u_0$ as above;
		\item  $\Omega = B_R(0)$ for some $R>0$. 
	\end{enumerate}
\end{theorem} 

\begin{theorem}[Centrosymmetry result]\label{thm:fw-centrosym}
	Let $s \in (0,1)$ and $w$ be the solution of \eqref{eq:fw}. Let $B_\delta(0)$ be a ball
	centered at the origin with radius $\delta > 0$ such that $B_\delta(0) \subset \Omega$. Let us consider the Laplace transform $W_\lambda(x) = \int_0^\infty e^{-\lambda t}w(x,t) \dd t$ (for $\lambda >0$). Assume that $\Omega$ is a $C^{1,1}$ bounded star-shaped domain with respect to the origin and that $u_0 \in C_0^\infty(\Omega)$, $\supp u_0 \subset B_\delta(0)$, and $u_0(x) = u_0(-x)$ for $x \in B_\delta(0)$. Then, the following conditions are equivalent: 
	\begin{enumerate}
	\item $\nabla_x W_\lambda(0) =0$ for any $\lambda >0$ and any $u_0$ as above;
		\item $\Omega$ is centrosymmetric with respect to the origin (i.e. $x \in \Omega \implies-x \in \Omega$).
	\end{enumerate} 
\end{theorem}

\begin{theorem}[Stationary zero points and a balance law for the Cauchy problem]\label{thm:1-z}
	Let $s \in (0,1)$, $u_0 \in C^\infty_0(\mathbb R^N)$ with $\supp(u_0) \subset B_L(0)$ for some $L >0$, and $w$ be the solution of the Cauchy problem for the fractional wave equation \eqref{eq:fwCauchy}. Then the following conditions are equivalent: 
	\begin{enumerate}
		\item $w(0,t) = 0$ for any $t >0$; 
		\item $\displaystyle \int_{\mathbb{S}^{N-1}}  u_0(r\omega) \dd \omega = 0$ for any $r \ge 0$. 
	\end{enumerate}
\end{theorem}

\begin{theorem}[Radial symmetry result]\label{thm:fw-sym-z}
	Let $s \in (0,1)$ and $w$ be the solution of \eqref{eq:fw}.  Let $B_\delta(0)$ be a ball
	centered at the origin with radius $\delta > 0$ such that $B_\delta(0) \subset \Omega$.  Let us consider the Laplace transform $W_\lambda(x) = \int_0^\infty e^{-\lambda t}w(x,t) \dd t$ (for $\lambda >0$). Assume that $\Omega$ is a $C^{1,1}$ bounded domain and that $u_0 \in C_0^\infty(\Omega)$,  $\supp u_0 \subset B_\delta(0)$, and that the balance law 
	\begin{align*}
	\int_{\mathbb{S}^{N-1}}  u_0(r\omega) \dd \omega = 0
	\end{align*}
	holds for any $r \in (0,\delta)$. 
	Then the following conditions are equivalent: 
	\begin{enumerate}
		\item $W_\lambda(0) =0$ for  any $\lambda >0$ and any $u_0$ as above;
		\item  $\Omega = B_R(0)$ for some $R>0$. 
	\end{enumerate}
\end{theorem}

\begin{theorem}[Centrosymmetry result]\label{thm:fw-centrosym-z}
	Let $s \in (0,1)$ and $w$ be the solution of \eqref{eq:fw}. Let $B_\delta(0)$ be a ball
	centered at the origin with radius $\delta > 0$ such that $B_\delta(0) \subset \Omega$.  Let us consider the Laplace transform $W_\lambda(x) = \int_0^\infty e^{-\lambda t}w(x,t) \dd t$ (for $\lambda >0$). Assume that $\Omega$ is a $C^{1,1}$ bounded star-shaped domain with respect to the origin and that  $u_0 \in C_0^\infty(\Omega)$, $\supp u_0 \subset B_\delta(0)$, and $u_0(x) = - u_0(-x)$ for $x \in B_\delta(0)$. Then, the following conditions are equivalent: 
	\begin{enumerate}
		\item  $W_\lambda(0) =0$ for  any $\lambda >0$ and any $u_0$ as above;
		\item $\Omega$ is centrosymmetric with respect to the origin (i.e. $x \in \Omega \implies-x \in \Omega$).
	\end{enumerate} 
\end{theorem}

\section{Proofs of the main results}
\label{sec:proofs}

\subsection{Critical points of the solution of the Cauchy problem}
\label{ssec:cauchy}



\begin{proof}[Proof of Theorem \ref{thm:1}]

%
	
	
	As a consequence of 
	\cite[Theorem 1.5]{MR3413864}, we have 
	$$ \partial_{x_j} P^{N}(x,t;s) = -2\pi x_j \, P^{N+2}(\tilde x,t;s),$$
	where $x = (x_1, \ldots, x_n) \in \mathbb R^N$, $\tilde x = (x_1, \ldots, x_n, 0, 0) \in \mathbb R^{N+2}$, and $P^{N}(x,t;s)$ is the heat kernel for $(-\Delta)^s$ in $N$ space dimensions. 
	Moreover, $P^{N+2}(\tilde x, 1;s) = \Phi_s(\tilde x)$ and, by \cite[Eq. (16.14)]{MR3916700}. We have 
	\begin{align}
	P^{N+2}(\tilde x, 1;s) = \frac{2\pi}{|x|^{\tfrac{N}{2}}} \int_0^\infty e^{-(2\pi \rho)^{2s}} \rho^{\tfrac{N}{2}+1} J_{\tfrac{N}{2} }(2\pi |x|\rho) \dd \rho,
	\end{align}
	where, by \cite[Eq. (4.23), p. 22]{MR3916700},  
	\begin{align}\label{eq:J}
	J_{\frac{N}{2}}(z) = \sum_{k=0}^\infty (-1)^k \frac{(z/2)^{N/2+2k}}{\Gamma(k+1)\Gamma(k+N/2+1)}, 
	\end{align}
	for $|z|< \infty$ and $|\arg z|< \pi$.
	For $u_0\in L^\infty(\mathbb R^N)$ with $\supp u_0 \subset B_L(0)$, we let 
	$$
	A(r)=\int_{\mathbb S^{N-1}}\omega u_0(r\omega)\dd\omega\quad \mbox{ for }r\in [0,L].
	$$
	Hence, we compute
	\begin{align*}
	\nabla_x u(0,t) &= 2\pi \int_{\R^N} x P^{N+2}(\tilde x, t;s) u_0(\tilde x) \dd \tilde x
	\\&=  2\pi  t^{-\frac{N+2}{2s}} \int_{\R^N} x P^{N+2}\left(\frac{\tilde x}{t^{\frac{1}{2s}}}, 1;s \right) u_0(\tilde x) \dd \tilde x
	\\&=  2\pi  t^{-\frac{N+2}{2s}} \int_0^L r^N P^{N+2}(hr,1;s) A(r) \dd r\\ 
	&=  (2\pi)^2  t^{-\frac{N+2}{2s}+
	\frac{N}{4s}} \int_0^L r^\frac{N}{2} \left(\int_0^\infty e^{-(2\pi \rho)^{2s}}\rho^{\tfrac{N}{2}+1} J_{\tfrac{N}{2} }(2\pi hr\rho) \dd \rho  \right) A(r) \dd r,
	\end{align*}
	where we set $h:= t^{-1/2s}$ and note that $0 < hr < hL$. Thus, by \eqref{eq:J}
	\begin{align*}
	&\int_0^\infty e^{-(2\pi \rho)^{2s}}\rho^{\tfrac{N}{2}+1} J_{\tfrac{N}{2} }(2\pi hr\rho) \dd \rho 
	\\&=(\pi h)^{N/2} r^{N/2} \sum_{k=0}^\infty (-1)^k \frac{1}{\Gamma(k+1)\Gamma(k+N/2+1)}(\pi h)^{2k} r^{2k}C_k, 
	\end{align*}
	where we set $C_k=\int_0^\infty e^{-(2\pi \rho)^{2s}} \rho^{N+1+2k} \dd \rho > 0.$
	Therefore, $\nabla_x u(0,t) = 0$ for all $t>0$ if and only if 
	\begin{align*}
\sum_{k=0}^\infty (\pi h)^{2k} (-1)^k \frac{C_k}{\Gamma(k+1)\Gamma(k+N/2+1)}\int_0^L r^{2k} A(r) r^N \dd r = 0  \text{ for all $h > 0$,} 
	\end{align*}
that is, 
\begin{align*}
\int_0^L r^{2k} (A(r) r^N) \dd r = 0 \quad \text{ for all $k \ge 0$,}
\end{align*}	
which is equivalent to 
\begin{align*}
A(r) = 0 \quad \text{ for any  $r \in [0,L]$.}	
\end{align*}

\end{proof}

\subsection{Radial symmetry result for the IBVP}
\label{ssec:radial_sym}

We now prove Theorem \ref{thm:sym}. For the first implication, we use an approach based on the Laplace transform and the properties of the operator $I + \lambda G$, with $\lambda >0$. For the second implication, we argue by contradiction by relying on the properties of the fundamental solution of the fractional heat equation. 


\begin{proof}[Proof of Theorem \ref{thm:sym}]
	\textbf{((2)$\implies$(1))} 
	\textbf{Step 1:} \emph{The Green's operator preserves the balance law.} Let us assume that $\Omega = B_R(0)$ and that the initial data $u_0$ satisfies the balance law.  As a first step towards the proof of the result, we show that $\int_{\S^{N-1}} \omega u_0 \dd \omega = 0$ implies  $\int_{\S^{N-1}} \omega Gu_0 \dd \omega = 0$, where $G$ is  the Green operator which sends $u_0$ into $Gu_0$ as in Remark \ref{remark meaning of Green operator}. To this end, we compute as follows (using the notation $G(x,y) = \hat G(|x-y|,|x|,|y|)\,$): for $x = \rho \omega$, and $0 < \rho \le R$, 
	\begin{align*}
	&\int_{\S^{N-1}} \omega (Gu_0)(\rho\omega) \dd \omega =\int_{\S^{N-1}} \omega \int_{B_R(0)} u_0(y) \hat G(\sqrt{\rho^2 + |y|^2 - 2 \rho \omega \cdot y}, \rho, |y|) \dd y \dd \omega \\&= \int_{B_R(0)} u_0(y) \int_{\S^{N-1}} \omega \hat G(\sqrt{\rho^2+|y|^2-2\rho \omega \cdot y}, \rho, |y|) \dd \omega \dd y\\&=
	\int_0^R s^{N-1} \dd s \int_{\S^{N-1}} \dd \eta \, u_0(s\eta) \int_{\S^{N-1}} \omega \hat G(\sqrt{\rho^2+s^2-2\rho s \omega \cdot \eta}, \rho, s) \dd \omega, 
	\end{align*}
	where we wrote $y = s\eta$, for $\eta \in \S^{N-1}$, $0 < s \le R$ in the last line. 
	
	Using the decomposition $\omega = (\omega \cdot \eta) \eta + \gamma$ with an integral identity of \cite[Eq. (1.2), p. 8]{MR2098409}, we have 
	\begin{align*}
	&\int_{\S^{N-1}} \omega \hat G(\sqrt{\rho^2 + s^2 - 2 \rho s \omega \cdot \eta}, \rho, s) \dd \omega
	\\&= \eta\int_{\S^{N-1}} (\omega \cdot\eta) \hat G(\sqrt{\rho^2 + s^2 - 2 \rho s \omega \cdot \eta}, \rho, s) \dd \omega
	\\&\quad + \underbrace{\int_{\S^{N-1}} \gamma \hat G(\sqrt{\rho^2 + s^2 - 2 \rho s \omega \cdot \eta}, \rho, s) \dd \omega}_{=0 \text{ by symmetry}}
	\\&= \eta |\S^{N-2}| \int_{-1}^1 (1-\lambda^2)^\frac{N-3}{2} \lambda \hat G(\sqrt{\rho^2 + s^2 - 2 \rho s \lambda}, \rho, s) \dd \lambda.
	\end{align*}
	As a result, 
	\begin{align*}
	&\int_{\S^{N-1}} u_0(s\eta) \eta |\S^{N-2}| \int_{-1}^{1} (1-\lambda^2)^{\frac{N-3}{2}}\lambda \hat G(\sqrt{\rho^2 + s^2 - 2 \rho s \lambda}, \rho, s) \dd \lambda \dd \eta \\
	&= |\S^{N-2}| \int_{-1}^{1} (1-\lambda^2)^{\frac{N-3}{2}} \lambda \hat G(\sqrt{\rho^2+s^2-2\rho s \lambda}, \rho, s) \dd \lambda \underbrace{\int_{\S^{N-1}} u_0(s\eta) \eta \dd \eta}_{= 0} 
	\\ &= 0,
	\end{align*}
	which yields that $\int_{\S^{N-1}} \omega (Gu_0)(\rho \omega) \dd \omega = 0$. 	
	
	\textbf{Step 2:} \emph{Reduction to an elliptic problem.}	For $\lambda >0$, we apply the Laplace transform to the solution of the fractional heat equation and obtain that $v_\lambda(x) = \int_0^\infty e^{-\lambda t}u(x,t) \dd t$ solves  
	\begin{align}\label{eq:fl-vl}
	\begin{cases}
	(-\Delta)^s v_\lambda(x) + \lambda v_\lambda(x) = u_0(x), & x \in \Omega, \\
	v_\lambda(x) = 0, & x \in \R^N \setminus \Omega.
	\end{cases}
	\end{align}
	Using the Green's operator of the fractional Laplacian on a ball, we deduce 
	\begin{align*}
	v_\lambda = G(u_0) - \lambda G v_\lambda, 
	\end{align*}
	that is, 
	\begin{align*}
	(I+\lambda G) v_\lambda = Gu_0. 
	\end{align*}
	
	\textbf{Step 3:} \emph{Balance law via Neumann series.} We claim that $(I + \lambda G)$ is invertible for $\lambda>0$ and
$
	v_\lambda = (I+\lambda G)^{-1} Gu_0.
$
	To prove this, we argue as follows. It follows from Proposition \ref{prop:estimate1} that there exists $\lambda_0 >0$ such that 
\begin{align*}
v_\lambda = (I+\lambda G)^{-1}Gu_0=\sum_{k=0}^\infty (-\lambda)^{k}G^{k+1}u_0
\end{align*}
for every $0 < \lambda \le \lambda_0$. Then, by Step 1, since $u_0$ satisfies the balance law, we have that $\int_{\S^{N-1}} \omega v_\lambda(\rho\omega)\dd\omega = 0$ for every $0 < \rho \le R$ and for every $0 <\lambda\le \lambda_0$.

 For $\lambda >0$, we write \eqref{eq:fl-vl} as 
	\begin{align}\label{eq:fl-vl2}
\begin{cases}
(-\Delta)^s v_\lambda(x) + \lambda_0 v_\lambda(x) = u_0 - (\lambda - \lambda_0)v_\lambda(x), & x \in \Omega, \\
v_\lambda(x) = 0, & x \in \R^N \setminus \Omega.
\end{cases}
\end{align}

Let $G_{\lambda_0}: u_0 \mapsto v_{\lambda_0}=:G_{\lambda_0} u_0$ where $v_{\lambda_0}$ is the solution of \eqref{eq:fl-vl} with $\lambda = \lambda_0$. Then we can write 
\begin{align*}
v_\lambda &= G_{\lambda_0}(u_0-(\lambda-\lambda_0)v_\lambda) \\
&=G_{\lambda_0} u_0 - (\lambda - \lambda_0) G_{\lambda_0}v_\lambda;
\end{align*}
hence
\begin{align*}
 (I+(\lambda - \lambda_0)G_{\lambda_0}) v_\lambda = G_{\lambda_0} u_0. 
\end{align*}
By Proposition \ref{prop:estimate2}, we have 
\begin{align*}
v_\lambda = (I+(\lambda-\lambda_0)G_{\lambda_0})^{-1} G_{\lambda_0} u_0 &= \sum_{k=0}^\infty (\lambda_0-\lambda)^{k} G_{\lambda_0}^{k+1} u_0
\end{align*}
for every $\lambda_0 < \lambda \le 2\lambda_0$. Then, since we have seen that $G_{\lambda_0} u_0$ satisfies the balance law, we have that $\int_{\S^{N-1}} \omega v_\lambda(\rho\omega)\dd\omega = 0$ for every $0 < \rho \le R$ and for every $\lambda_0 < \lambda \le 2 \lambda_0$. Iterating the argument, eventually we obtain that the solution $v$ of \eqref{eq:fl-vl} with $\lambda >0$ preserves the balance law: i.e.,  
\begin{align*}
\int_{\S^{N-1}} \omega v_\lambda(\rho\omega) \dd \omega = 0.
\end{align*}

\textbf{Step 4:} \emph{Conclusion of the proof.} Assuming $\int_{\S^{N-1}} \omega u_0 \dd \omega = 0$, Steps 1-3 give that, for every $0 < \rho \le R$ and every $\lambda >0$, 
\begin{align*}
\int_{\S^{N-1}} \omega v_\lambda(\rho\omega ) \dd \omega = 0,
\end{align*}
which shows, by the injectivity of the Laplace transform, that 
\begin{align*}
\int_{\S^{N-1}} \omega u(\rho\omega,t) \dd \omega = 0,
\end{align*}
	for every $0 < \rho \le R$ and $t \ge 0$. 
 
 	\textbf{((1)$\implies$(2))} \textbf{Step 1:} \emph{Choice of initial data and symmetry properties of $\nabla_x G(0,y)$.} Let 
	\begin{align}\label{function v}
	v(x)=\int_0^\infty u(x,t) \dd t.
	\end{align}
	Then $v$ solves
	\begin{align}\label{eq:fl-vl-zero-lambda}
	\begin{cases}
	(-\Delta)^s v(x)  = u_0(x), & x \in \Omega, \\
	v(x) = 0, & x \in \R^N \setminus \Omega.
	\end{cases}
	\end{align}
	Let $G=G(x,y)$ be the Green's function of the fractional Laplacian on $\Omega$. Then 
	\begin{align}\label{representation by Green's function of v}
	v(x)=\int_{B_\delta(0)}G(x,y) u_0(y) dy.
	\end{align}
	Let us consider as initial data $u_0(x) = \eta(|x|)\psi(x/|x|)$, where $\eta:\R \to \R$ is a smooth function with support in $(0,\delta)$ and $\psi$ on $\S^{N-1}$ satisfies the balance law $\int_{\S^{N-1}} \omega\psi(\omega) \dd \omega = 0$. 
	Then, since $\nabla_x u(0,t)=0$ for any $t > 0$,  
	\begin{align*}
	0 &= \int_{B_\delta(0)} \nabla_x G(0,y) \eta(|y|)\psi(y/|y|) \dd y \\
	&=\int_0^\delta \eta(r)\left(r^{N-1} \int_{\S^{N-1}} \nabla_x G(0,r\omega) \psi(\omega) \dd \omega \right) \dd r,
	\end{align*}
	which implies 
	\begin{align*}
\int_{\S^{N-1}} \nabla_x G(0,r\omega) \psi(\omega)  \dd \omega = 0 \text{ for any $r \in (0,\delta)$},
\end{align*}
and hence
\begin{align} \label{eq:M(r)}
\nabla_x G(0,y) = M(r) \omega \quad  \text{ for any  $y = r\omega \in B_\delta(0) \setminus \{0\}$},
\end{align}
where $M(r)$ is a $N\times N$ matrix-valued function in $r=|y|$ and $\omega \in \S^{N-1}$.
Note that $\nabla_x  G(0,y)$ is real-analytic in $y \in \Omega \setminus \{0\}$ (by Proposition \ref{lm:prob-formula}) and $\nabla_x G(0,y) \equiv 0$ for every $y \in \R^N \setminus \Omega$. 

\textbf{Step 2:} \emph{Setting up the argument by contradiction.} 
For every direction $\omega \in \S^{N-1}$, there exists $R(\omega) >\delta$ such that the line segment $\ell_\omega := \{r\omega \in \R^N: 0 \le r < R(\omega)\}$ is contained in $\Omega$ and $R(\omega)\omega \in \partial \Omega$. Moreover, $\nabla_x G(0,r\omega)$ is real analytic in $r \in (0,R(\omega))$ and $\nabla_x G(0,R(\omega)\omega) = 0$ for every $\omega \in \S^{N-1}$. Let $B_{R^*}(0) \subset \Omega$ and $P \in \bar B_{R^*}(0) \cap \partial \Omega $ for some $P=R^*\omega^* = R(\omega^*) \omega^* \in \partial \Omega$. We claim that $B_{R^*}(0) = \Omega$. Let us suppose, for the sake of finding a contradiction, that $B_{R^*} \subsetneqq \Omega$. 

\textbf{Step 3:} \emph{A construction to represent $\nabla_x G(0,y)$.} 
Since $B_{R^*} \subsetneqq \Omega$, there exists a ball $B_\eps(Q)$ such that $\bar B_{\eps}(Q) \subset \Omega$ for some $Q \in \partial B_{R^*}(0)\cap \Omega$ and $\eps >0$  (to be chosen small enough). Let us choose a set of linearly independent vectors $f_1, \dots, f_N \in \S^{N-1}$ satisfying 
\begin{align}
(R^* + \tfrac{\eps}{2})f_i \in \partial B_{R^* +\eps/2}(0) \cap \partial B_\eps(Q), \quad i = 1, \dots, N.
\end{align} 
Then every $\omega \in \S^{N-1}$ is represented by $ \omega = \sum_{i=1}^N \eta_i f_i$ for a unique $(\eta_1,\cdots,\eta_N)\in\mathbb R^N$.
If $0 < r < \delta$, then
\begin{align*}
\nabla_x G(0,r\omega) = M(r) \omega &= \sum_{i=1}^N \eta_i M(r) f_i = \sum_{i=1}^N \eta_i \nabla_x G(0,rf_i),
\end{align*}
namely \begin{align}\label{eq:G2}
\nabla_x G(0,r\omega) = \sum_{i=1}^N \eta_i \nabla_x G(0,rf_i)
\end{align}
for every $0 < r < \delta$ and $\omega \in \S^{N-1}$. By the real analyticity of $\nabla_x G(0,y)$ in $y$, we have 
\begin{align}
\nabla_xG(0,r\omega) = \sum_{i=1}^N \eta_i \nabla_x G(0,rf_i)
\end{align}\label{eq:nablaG}
for every $0 < r < \min\{R(\omega), R^* +\eps/2\}$ and $\omega = \sum_{i=1}^N \eta_i f_i\in \S^{N-1}$.

Set $h_i=h_i(r) = \nabla_x G(0,r f_i)$ for $i=1, \dots, N$. Then every $h_i$ is real analytic in $r \in (0,R^* + \eps/2)$.

Hence,  by \eqref{eq:G2}, we have 

\begin{align}\label{eq:G3}
\nabla_x G(0,r\omega) = \sum_{i=1}^N \eta_i h_i(r)
\end{align}
for every $r\omega \in C$, where $C$ denotes the connected component of $\Omega \cap B_{R^*+\frac \varepsilon2}(0)$ containing $B_{R^*}(0)$. 

Consider the set $F:= \partial \Omega \cap C$. Since $\nabla_x G(0,R(\omega)\omega) = 0$ for every $\omega \in \S^{N-1}$, from \eqref{eq:G3} we deduce 
\begin{align}
\sum_{i=1}^N \eta_i h_i(R(\omega))=0, \quad R(\omega)\omega \in F.
\end{align}
Since $(\eta_1,\cdots,\eta_N)\not=0$, we get
\begin{align}
\det [h_1(R(\omega)) \cdots h_N(R(\omega))] = 0, \quad R(\omega)\omega \in F,
\end{align}
We remark that the range of $R(\omega)$ contains $(R^*,R^* + \varepsilon^*)$ for some $\varepsilon^*>0$. Indeed, since $B_{R^*}(0) \subsetneqq \Omega$ and $\partial\Omega\cap\partial B_{R^*}(0)\not=\emptyset$, there must be a connected component  $\Gamma \subsetneqq \partial \Omega$ intersecting $\partial B_{R^*}(0)$ and containing some point $Q^* \in \Gamma \setminus \partial B_{R^*}(0)$; then we set $\varepsilon^*=\operatorname{dist}(Q^*,0)  - R^*>0$. 

\textbf{Step 4:} \emph{Conclusion of the argument.}
This fact and the analyticity of $h_i$ imply 
\begin{align}
\det[h_1(r) \dots h_N(r)] = 0 \quad 0 < r < R^* + \eps/2.
\end{align}
In particular, for $0 < r < \delta$, since by \eqref{eq:M(r)} $\nabla_x G(0,rf_i) = M(r)f_i$, with $i=1, \dots, N$, we have
\begin{align*}
0 \equiv \det[h_1(r) \cdots h_N(r)] &= \det[M(r) f_1 \cdots M(r) f_N] \\&= \det M(r) \det[f_1 \cdots f_N], \quad  0 < r < \delta.
\end{align*}
Since $\det[f_1 \cdots f_N]\not=0$, we conclude that $\det M(r)\equiv 0$ for  $0 < r < \delta$.
However, for $y \in B_\delta(0) \setminus \{0\}$, we have 
\begin{align}\label{Green function singularity}
\nabla_x G(0,y) \sim C r^{-N + 2s -1}\omega, 
\end{align}
with $C \neq 0$, $y = r\omega$, $\omega \in \S^{N-1}$, as $r \to 0^+$. 
This contradicts \eqref{eq:M(r)}.

	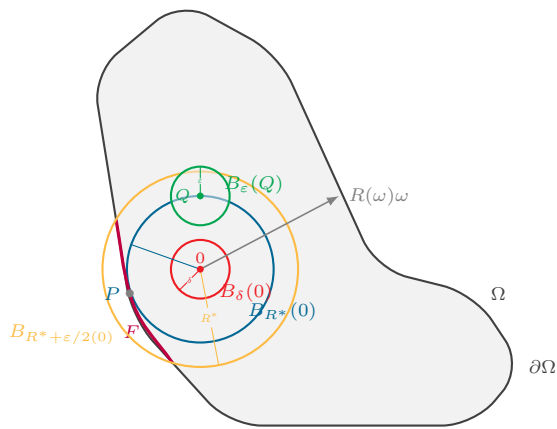
\begin{figure}[h]
		\centering
	\begin{tikzpicture}[thick,black!75,scale=1.3]
	\draw[name path=blackRegion,
	rounded corners=12pt
	,	fill=black!5
	] (0,0) -- (-0.4,2.5)
	-- (0.5,3.25) 
	-- (1.5,2.7)
	-- (2.5,0.5)
	--(3.5,0.25) 
	--(4,-0.5) 
	-- (3.5,-1.1)
	-- (1,-1.1) -- cycle;
	
	\path[name path=YellowRegion] (0.7,0.5) circle(1cm);

	\begin{scope}[xshift=-0.2cm]
	\draw[MidnightBlue,fill=black!5] (0.9,0.5) circle(0.75cm) 
	node[below right=0.3cm and 0.5cm]{\scriptsize $B_{R^*}(0)$};
	\draw [very thick,
	name intersections={of=blackRegion and YellowRegion},
	purple] (intersection-1) to[out=280,in=106]  (0.18,0.25)
	to[out=285,in=125] (intersection-2)
	node[above left=0.2cm and 0.3cm]{\scriptsize $F$};
	
	\draw[Dandelion] (0.9,0.5) circle(1cm)
	node[below left=0.6cm and 1cm]{\scriptsize $B_{R^* + \eps/2(0)}$};
	
	\draw[Red,fill=black!5] (0.9,0.5) circle(0.3cm)
	node[below right=0.051cm and 0.12cm]{\scriptsize $B_\delta(0)$};

	\node[gray] at (0.18,0.25){\scriptsize $\bullet$};
	\node[left,MidnightBlue] at (0.18,0.25){\scriptsize$P$};
	
	\draw[-latex,semithick,gray] (0.9,0.5) -- ++(1.42,0.75)
	node[right,black!50]{\scriptsize $R(\omega)\omega$};
	
	\draw[Red,thin] (0.9,0.5) -- ++(-0.22,-0.22)
	node[midway,
	scale=0.5,
	fill=black!5,
	circle,
	inner sep=0pt]{\tiny $\delta$};
	
	\draw[Dandelion,thin] (0.9,0.5) -- ++(0.19,-0.99)
	node[midway,
	scale=0.7,
	fill=black!5,
	circle,
	inner sep=0pt]{\tiny $R^*$};
	
	\draw[MidnightBlue,thin] (0.9,0.5) -- ++(-0.7,0.25)
	node[pos=0.7,
	scale=0.7,
	circle,
	inner sep=0pt]{\scriptsize};  
	
	\fill[Red] (0.9,0.5) circle(1pt)
	node[above=-0.05]{\tiny $0$};
	
	\end{scope}

	\begin{scope}[xshift=-0.2cm,yshift=0.75cm]
	\draw[Green,
	fill opacity=0.5,
	fill=Black!5] (0.9,0.5) circle(0.3cm)
	node[above right=-0.1cm and 0.185cm,opacity=1]{\scriptsize $B_\eps(Q)$};
	\fill[Green] (0.9,0.5)circle(1pt);
	\node[left,Green] at (0.9,0.5){\tiny $Q$};
	
	\draw[Green,thin] (0.9,0.5) -- ++(0,0.3)
	node[midway,
	scale=0.5,
	fill=black!5,
	circle,
	inner sep=0pt]{\tiny $\eps$};
	\end{scope}
	
	
	\node at (3.75,0.25){\scriptsize$ \Omega$};
	\node[inner sep=0pt] (A2) at (4.2,-0.5){\scriptsize$\partial \Omega$};
	\end{tikzpicture}
	\caption{Illustration of the construction in the proof of Theorem \ref{thm:sym}.}
	\label{fig:sym}
\end{figure}	
	\end{proof}

\subsection{Centro-symmetry result for the IBVP}
\label{ssec:centro-sym}

In this section, we prove Theorem \ref{thm:centrosym}. The first implication follows from the uniqueness of solutions to \eqref{eq:fhe-bc}. In the second one, which uses the assumption on $\Omega$ being star-shaped, we rely on the unique continuation properties for the fractional Laplacian.

\begin{proof}[Proof of Theorem \ref{thm:centrosym}]
	\textbf{((2)$\implies$(1))} Let $w(x,t) = u(-x,t)$. Then $w$ is also a solution of \eqref{eq:fhe-bc}. Uniqueness of the solution gives $w = u$; hence, in particular, $\nabla u(0,t) = 0$ for  $t>0$.
	
	\textbf{((1)$\implies$(2))} \textbf{Step 1:} \emph{Reduction to an elliptic problem.} By employing the function $v$ given by \eqref{function v} as in \textbf{((1)$\implies$(2))} \textbf{Step 1} of the proof of Theorem \ref{thm:sym}, we have  \eqref{eq:fl-vl-zero-lambda} and \eqref{representation by Green's function of v}.

\textbf{Step 2:} \emph{Choice of initial data and properties of the Green's function.}	For any $\psi \in C^\infty_0(B_\delta(0))$, we write 
	$$
	u_0(x) = \psi(x) + \psi(-x).
	$$
	Then $u_0(x) = u_0(-x)$ and, by assumption, 
	\begin{align*}
	0 &= \nabla v(0) = \int_{B_\delta(0)} \nabla_x G(0,y) (\psi(y)+\psi(-y)) \dd y
	\\&= \int_{B_\delta(0)} (\nabla_x G(0,y) + \nabla_x G(0,-y)) \psi(y) \dd y. 
	\end{align*}
	Therefore, since $\psi \in C^\infty_0(B_\delta(0))$ is arbitrarily chosen, we have 
	\begin{align*}
	\nabla_x G(0,y) + \nabla_x G(0,-y) \equiv 0, \quad y \in B_\delta(0).
	\end{align*}
	
	\textbf{Step 3:} \emph{Reflection and unique continuation arguments.} Let us consider the reflected domain $\Omega^* = \{x \in \R^N: -x \in \Omega\}$ and let $C$ be the connected component of $\Omega \cap \Omega^\ast$ containing the origin. Since $\Omega$ is star-shaped with respect to the origin, we actually have $C = \Omega \cap  \Omega^*$.
	
	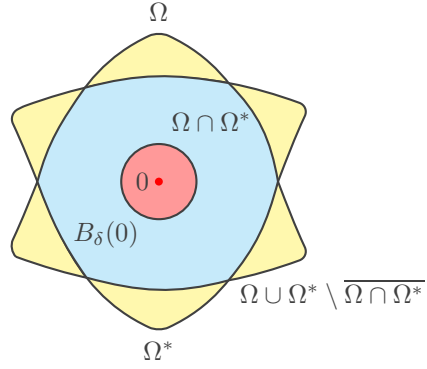
\begin{figure}[h]
		\centering
		\begin{tikzpicture}[thick,black!75]
		
		\begin{scope}[rounded corners]
		\draw[name path=triangle1,fill=yellow!40] (-1.9,0) to[out=-20,in=200] (2,0)
		to[out=115,in=-25] (0,3)
		node[above]{$\Omega$} 
		to[out=205,in=65] (-2,0.1) -- (-1.9,0) ;
		
		\draw[name path=triangle2,fill=yellow!40] (-1.8,2) to[out=20,in=160] (2,2)
		to[out=245,in=25] (0,-1) 
		node[below]{$\Omega^*$} 
		to[out=155,in=-65] (-2,1.9) -- (-1.8,2); 
		
		\path [name intersections={of=triangle1 and triangle2}] (intersection-1) -- (intersection-2);
		\end{scope}
		
		\draw [fill=cyan!20] (intersection-1) to[out=195,in=-15] (intersection-2)
		to[out=125,in=-70] (intersection-6)
		to[out=75,in=235] (intersection-5)
		to[out=15,in=170] (intersection-3)
		to[out=-45,in=100] (intersection-4)
		to[out=255,in=50] cycle;
		
		\draw[fill=red!40] (0,1) circle(0.5cm)node[left]{$0$};
		\fill[red] (0,1) circle(1.5pt);
		
		\node at (-0.7,0.3){$B_\delta(0)$};
		\node at (0.7,1.8){$\Omega \cap \Omega^*$};
		\node at (2.3,-0.5){$\Omega \cup \Omega^* \setminus \overline{\Omega \cap \Omega^*}$};
		
		\end{tikzpicture}
		\caption{Illustration of the intersection $\Omega \cap \Omega^*$ used in the proof of Theorem \ref{thm:centrosym}.}
		\label{fig:centros}
	\end{figure}

	By the real-analyticity of $G(x,y)$ in $y$, we have that 
	\begin{align*}
	h(y) := \nabla_x G(0,y) + \nabla_x G(0,-y) \equiv 0, \quad y \in C.
	\end{align*}
	
	Suppose, for the sake of finding a contradiction, that $\Omega \neq \Omega^*$. Then $D= \Omega \setminus \bar C \neq \emptyset$. Using the mean-value formula \eqref{eq:mean-value} and the observation in \cite[p. 21]{MR2569321} (on $s$-harmonic equations on the union of two domains), we get that $h$ also satisfies 
	\begin{align*}
	\begin{cases}
	(-\Delta)^s h = 0, & y \in (\Omega \cup \Omega^*) \setminus (\overline{\Omega \cap \Omega^*}), \\
	h = 0, & y \in \R^N \setminus ((\Omega \cup \Omega^*) \setminus (\overline{\Omega \cap \Omega^*})).
	\end{cases}
	\end{align*}
	Then $h \equiv 0$ and in particular $\nabla_x G(0,y)=h(y)=0$ for $y\in D$, which implies that $\nabla_x G(0,y)\equiv 0$ for $y\in \Omega\setminus\{0\}$ by the real-analyticity. This contradicts the fact that $G(0,y)$ has a singularity at the origin as is mentioned in \eqref{Green function singularity}.
\end{proof}


\subsection{Spatial zero points}
\label{sec:proofs-z}


Let us consider the problems for spatial zero points. Namely, instead of $\nabla u(\cdot, t)=0$, we consider $u(\cdot, t)=0$ for each time $t\ge 0$. Then, along the similar arguments we can get all theorems by replacing the balance law $\int_{\mathbb{S}^{N-1}} \omega u_0(\omega r) \dd \omega = 0$ by $\int_{\mathbb{S}^{N-1}}  u_0(\omega r) \dd \omega = 0$. For the sake of completeness, here we outline the proof of Theorems \ref{thm:1-z}, \ref{thm:sym-z} and \ref{thm:centrosym-z}.

\begin{proof}[Proof of Theorem \ref{thm:1-z}]
	
	As in the proof of Theorem \ref{thm:1}, we have $P^{N}(x, 1;s) = \Phi_s(x)$ and, by \cite[Eq. (16.14)]{MR3916700}. We have 
	\begin{align}
	P^{N}(x, 1;s) = \frac{2\pi}{|x|^{\tfrac{N}{2}-1}} \int_0^\infty e^{-(2\pi \rho)^{2s}} \rho^{\tfrac{N}{2}} J_{\tfrac{N}{2}-1 }(2\pi |x|\rho) \dd \rho.
	\end{align}  
	Let $u_0\in L^\infty(\mathbb R^N)$ with $\supp u_0 \subset B_L(0)$. Set 
	$$
	a(r)=\int_{\mathbb S^{N-1}} u_0(r\omega)\dd\omega\mbox{ for }r\in [0,L].
	$$
	Hence, we compute
	\begin{align*}
	u(0,t) &= 2\pi \int_{\R^N}  P^{N}(x, t;s) u_0( x) \dd  x
	\\&=  2\pi  t^{-\frac{N}{2s}} \int_{\R^N}  P^{N}\left(\frac{ x}{t^{\frac{1}{2s}}}, 1;s \right) u_0( x) \dd  x
	\\&=  2\pi  t^{-\frac{N}{2s}} \int_0^L r^{N-1} P^{N}( hr,1;s)  a(r) \dd r \\ 
	&=  (2\pi)^2  t^{-\frac{N+1}{2s}+
		\frac{N}{4s}} \int_0^L r^{\frac{N}{2}} \left(\int_0^\infty e^{-(2\pi \rho)^{2s}}\rho^{\tfrac{N}{2}} J_{\tfrac{N}{2}-1 }(2\pi hr\rho) \dd \rho  \right) a(r) \dd r,
	\end{align*}
	where we set $h:= t^{-1/2s}$ and note that $0 < hr < hL$.
	By \cite[Eq. (4.23), p. 22]{MR3916700}, we have 
	\begin{align*}
	J_{\frac{N}{2}-1}(z) = \sum_{k=0}^\infty (-1)^k \frac{(z/2)^{N/2-1+2k}}{\Gamma(k+1)\Gamma(k+N/2)}, 
	\end{align*}
	with $|z|< \infty$ and $|\arg z|< \pi$. Hence,
	\begin{align*}
	&\int_0^\infty e^{-(2\pi \rho)^{2s}}\rho^{\tfrac{N}{2}} J_{\tfrac{N}{2}-1 }(2\pi hr\rho) \dd \rho 
	\\&=(\pi h r)^{N/2-1} \sum_{k=0}^\infty (-1)^k \frac{1}{\Gamma(k+1)\Gamma(k+N/2)}(\pi h)^{2k} r^{2k} c_k, 
	\end{align*}
	where we set $c_k=\int_0^\infty e^{-(2\pi \rho)^{2s}} \rho^{N-1+2k} \dd \rho$.
	Therefore $u(0,t) = 0$ for all $t>0$ if and only if 
	\begin{align*}
	\sum_{k=0}^\infty (\pi h)^{2k} (-1)^k \frac{c_k}{\Gamma(k+1)\Gamma(k+N/2)}\int_0^L r^{2k} a(r) r^{N-2} \dd r = 0  \text{ for all $h > 0$,} 
	\end{align*}
	that is, 
	\begin{align*}
	\int_0^L r^{2k} (a(r) r^{N-1}) \dd r = 0 \text{ for all $k \ge 0$,}
	\end{align*}	
	which is equivalent to 
	\begin{align*}
	a(r) = 0 \text{ for any  $r \in [0,L]$.}	
	\end{align*}
	
\end{proof}


\begin{proof}[Proof of Theorem \ref{thm:sym-z}]
	\textbf{((2)$\implies$(1))} 
\textbf{Step 1:} \emph{The Green's function preserves the balance law.} Let us assume that $\Omega = B_R(0)$ and that the initial data $u_0$ satisfies the balance law.  As a first step towards the proof of the result, we show that $\int_{\S^{N-1}} u_0 \dd \omega = 0$ implies  $\int_{\S^{N-1}} Gu_0 \dd \omega = 0$, where $G$ is  the Green's operator which sends $u_0$ into $Gu_0$ as in Remark \ref{remark meaning of Green operator}. To this end, we compute as follows (using the notation $G(x,y) = \hat G(|x-y|,|x|,|y|)\,$): for $x = \rho \omega$, and $0 < \rho \le R$, 
\begin{align*}
&\int_{\S^{N-1}} (Gu_0)(\rho\omega) \dd \omega =\int_{\S^{N-1}}  \int_{\S^{N-1}} u_0(y) \hat G(\sqrt{\rho^2 + |y|^2 - 2 \rho \omega \cdot y}, \rho, |y|) \dd y \dd \omega \\&= \int_{\S^{N-1}} u_0(y) \int_{\S^{N-1}}  \hat G(\sqrt{\rho^2+|y|^2-2\rho \omega \cdot y}, \rho, |y|) \dd \omega \dd y\\&=
\int_0^R s^{N-1} \dd s \int_{\S^{N-1}} \dd \eta \, u_0(s\eta) \int_{\S^{N-1}}  \hat G(\sqrt{\rho^2+s^2-2\rho s \omega \cdot \eta}, \rho, s) \dd \omega, 
\end{align*}
where we wrote $y = s\eta$, for $\eta \in \S^{N-1}$, $0 < s \le r$ in the last line. 

Using an integral identity of \cite[Eq. (1.2), p. 8]{MR2098409}, we have 

\begin{align*}
	&\int_{\S^{N-1}} \hat G(\sqrt{\rho^2 + s^2 - 2 \rho s \omega \cdot \eta}, \rho, s) \dd \omega
	\\&=  |\S^{N-2}| \int_{-1}^1 (1-\lambda^2)^\frac{N-3}{2}  \hat G(\sqrt{\rho^2 + s^2 - 2 \rho s \lambda}, \rho, s) \dd \lambda.
	\end{align*}
	As a result, 
	\begin{align*}
	&\int_{\S^{N-1}} u_0(s\eta)  |\S^{N-2}| \int_{-1}^{1} (1-\lambda^2)^{\frac{N-3}{2}} \hat G(\sqrt{\rho^2 + s^2 - 2 \rho s \lambda}, \rho, s) \dd \lambda \dd \eta \\
	&= |\S^{N-2}| \int_{-1}^{1} (1-\lambda^2)^{\frac{N-3}{2}}  \hat G(\sqrt{\rho^2+s^2-2\rho s \lambda}, \rho, s) \dd \lambda \underbrace{\int_{\S^{N-1}} u_0(s\eta) \dd \eta}_{= 0} 
	\\ &= 0,
	\end{align*}
	which yields  that $\int_{\S^{N-1}} (Gu_0)(\rho \omega) \dd \omega = 0$. 


\textbf{Steps 2 -- 4:} Once we know that the Green's operator preserves the balance law in Step 1, the rest follows as in the corresponding steps in the proof of Theorem \ref{thm:sym}.

\textbf{((1)$\implies$(2))} \textbf{Step 1:} \emph{Choice of initial data and symmetry properties of $G(0,y)$.} Let us consider as initial data $u_0(x) = \eta(|x|)\psi(x/|x|)$, where $\eta:\R \to \R$ is a smooth function with support in $(0,\delta)$ and $\psi$ on $\S^{N-1}$ satisfies the balance law $\int_{\S^{N-1}} \psi(\omega) \dd \omega = 0$. 
Then, by employing the function $v$ given by \eqref{function v} as in \textbf{((1)$\implies$(2))} \textbf{Step 1} of the proof of Theorem \ref{thm:sym},  since $u(0,t)=0$ for any $t>0$, we obtain
\begin{align*}
0 &= \int_{B_\delta(0)} G(0,y) \eta(|y|)\psi(y/|y|) \dd y \\
&=\int_0^\delta \eta(r)\left(r^{N-1} \int_{\S^{N-1}}  G(0,r\omega) \psi(\omega) \dd \omega \right) \dd r,
\end{align*}
which implies 
\begin{align*}
\int_{\S^{N-1}}  G(0,r\omega) \psi(\omega)  \dd \omega = 0 \quad \text{ for any $r \in (0,\delta)$},
\end{align*}
and hence
\begin{align} \label{eq:M(r)-z}
G(0,y) = m(r)  \quad  \text{ for any $y = r\omega \in B_\delta(0) \setminus \{0\}$},
\end{align}
where $m(r)$ is a function in $r=|y|$ and $\omega \in \S^{N-1}$. 

\textbf{Step 2:} \emph{Conclusion of the proof.} Let $B_{R^*}(0) \subset \Omega$ and $P \in \bar B_{R^*}(0) \cap \partial \Omega $ for some $P\in \partial \Omega$ and $R^*>0$. Since $ G(0,y)$ is real-analytic in $y \in \Omega \setminus \{0\}$ by Proposition \ref{lm:prob-formula},  \eqref{eq:M(r)-z} yields that $G(0,y)$ is radially symmetric in $y \in \bar B_{R^*}(0)$. By observing that  $G(0,\cdot) >0$ in $\Omega$ and $G(0,\cdot) = 0$ in $\R^N \setminus \Omega$, we conclude that $\Omega = B_{R^*}(0)$.

\end{proof}


\begin{proof}[Proof of Theorem \ref{thm:centrosym-z}]
	\textbf{((2)$\impliedby$(1)) }Let $w(x,t) = -u(-x,t)$. Then $w$ is also a solution of \eqref{eq:fhe-bc}. Uniqueness of the solution gives $w = u$; hence, in particular, $u(0,t) = 0$ for  $t>0$.
	
\textbf{((1)$\implies$(2))} \textbf{Step 1:} \emph{Reduction to an elliptic problem.} Let us consider $v(x) = \int_0^\infty u(x,t) \dd t$ 
as in the proof of Theorem \ref{thm:centrosym-z}.
	
	\textbf{Step 2:} \emph{Choice of initial data and properties of the Green's function.} For any $\psi \in C^\infty_0(B_\delta(0))$, we write 
	$$
	u_0(x) = \psi(x) - \psi(-x).
	$$
	Then $u_0(x) = - u_0(-x)$ and, by assumption, 
	\begin{align*}
	0 &=  v(0) = \int_{B_\delta(0)}  G(0,y) (\psi(y)-\psi(-y)) \dd y
	\\&= \int_{B_\delta(0)} ( G(0,y) - G(0,-y)) \psi(y) \dd y. 
	\end{align*}
	Therefore, since $\psi \in C^\infty_0(B_\delta(0))$ is arbitrarily chosen, we have 
	\begin{align*}
	 G(0,y) -  G(0,-y) \equiv 0, \quad y \in B_\delta(0).
	\end{align*}
	
	\textbf{Step 3:} \emph{Reflection and unique continuation arguments.} Let us consider the reflected domain $\Omega^* = \{x \in \R^N: -x \in \Omega\}$ and let $C$ be the connected component of $\Omega \cap \Omega^\ast$ containing the origin. Since $\Omega$ is star-shaped, we actually have $C = \Omega \cap  \Omega^*$.
		
	By the real-analyticity of $G(x,y)$ in $y$, we have that 
	\begin{align*}
	h(y) := G(0,y) - G(0,-y) \equiv 0, \quad y \in C.
	\end{align*}
	
	Suppose, for the sake of finding a contradiction, that $\Omega \neq \Omega^*$. Then $D= \Omega \setminus \bar C \neq \emptyset$. We note that $h$ also satisfies 
	\begin{align*}
	\begin{cases}
	(-\Delta)^s h(x) = 0, & y \in (\Omega \cup \Omega^*) \setminus (\overline{\Omega \cap \Omega^*}), \\
	h(x) = 0, & y \in \R^N \setminus ((\Omega \cup \Omega^*) \setminus (\overline{\Omega \cap \Omega^*})).
	\end{cases}
	\end{align*}
	Then $h \equiv 0$ and in particular $G(0,y)=h(y)=0$ for $y\in D$, which implies that $G(0,y)\equiv 0$ for $y\in \Omega\setminus\{0\}$ by the real-analyticity. This contradicts the fact that $G(0,y)$ has a singularity at the origin as is mentioned in \eqref{Green function singularity}.
\end{proof}


\subsection{Fractional wave equation}
\label{sec:proofs-w}

For the Cauchy problem \eqref{eq:fwCauchy}, using the Fourier transform, we can prove a well-posedness result as in \cite{MR4043580}. From the representation formula in \cite{MR4043580}, it is also possible to deduce information on the regularity of the solution.

On the other hand,  the IBVP \eqref{eq:fw} is solved by the method of separation of variables as problem  \eqref{eq:fhe-bc} is solved in \cite{MR3462074} with the aid of the eigenfunctions of the fractional Laplacian. See also \cite[Theorem 2.1]{2105.11324} for more general settings by Galerkin's method.

To gain further regularity for the solution of \eqref{eq:fw}, we would need to take the initial data in the space
$$H^{s,\sigma}(\Omega) := \left\{u\in L^2(\Omega): \ \sum_{k \in \N}|\lambda_k^\sigma\langle u, \phi_k \rangle_{L^2(\Omega)}|^2 < + \infty \right\},$$
(for some suitable $\sigma >1$) where $\{\lambda_k\}_{k \in \N}$ is the (non-decreasing) sequence of eigenvalues of $(-\Delta)^s$ and $\{\phi_k\}_{k \in \N}$ is the corresponding sequence of eigenfunctions, where the fact that each $\phi_k \in C^\infty(\Omega)$ follows from the $L^\infty$ estimates \cite[Proposition 3.2]{MR3462074} and the bootstrap argument of \cite{MR3331523}.
Note that $H^{s,\sigma}(\Omega)$ is the domain of the $\sigma$-power of $(-\Delta)^s$. 

On the other hand, for the sake of proving the symmetry results, we shall only rely on the properties of the corresponding elliptic problem obtained through the Laplace transform. This, together with the fact that it is unclear whether or not a sufficient number of initial data in $H^{s,\sigma}(\Omega)$ satisfy the balance law \eqref{eq:balance-law} in $\Omega$, motives the assumptions in the Theorems for the fractional wave equation. 

With these considerations, all the symmetry results on the fractional wave equation follow from the ones for the heat equation thanks to the following two lemmas (which was proved in \cite[pp. 251-252]{MR847715} in case $s=1$).

\begin{lemma}[Relationship between fractional heat and wave equation for the IBVP]
	Let $u$ be the  solution of \eqref{eq:fhe-bc} and $w$ be the solution of \eqref{eq:fw} for $u_0 \in C^\infty_0(\Omega)$. Then the following holds: 
	\begin{enumerate}
		\item $W_\lambda(0) = 0$ for any $\lambda>0$ if and only if $u(0,t) = 0$ for any $t>0$; 
		\item $\nabla_x W_\lambda(0) = 0$ for any $\lambda>0$ if and only if $\nabla_x u(0,t) = 0$ for any $t>0$.
	\end{enumerate}
\end{lemma}

\begin{proof}
Let us consider the Laplace transform of $u$ and $w$: 
$U_\lambda(x) = \int_0^\infty e^{-\lambda t}u(x,t) \dd t$ and $W_\lambda(x) = \int_0^\infty e^{-\lambda t}w(x,t) \dd t$, which solve
\begin{align}\label{eq:lap2}
\begin{cases}
(-\Delta)^s W_\lambda(x) + \lambda^2 W_\lambda(x) = u_0(x), & x \in \Omega, \\
W_\lambda(x) = 0, & x \in \R^N \setminus \Omega.
\end{cases}
\end{align}
and
\begin{align}\label{eq:lap3}
\begin{cases}
(-\Delta)^s U_\lambda(x) + \lambda U_\lambda(x) = u_0(x), & x \in \Omega, \\
U_\lambda(x) = 0, & x \in \R^N \setminus \Omega.
\end{cases}
\end{align}
respectively. From the uniqueness of the elliptic boundary value problem, we deduce that $U_{\lambda^2} \equiv W_{\lambda}$.
Therefore, we have 
\begin{align*}
W_\lambda(0)&= \int_0^\infty e^{-\lambda^2 t}u(0,t) \dd t, \\
\nabla_x W_\lambda(0)  &= \int_0^\infty e^{-\lambda^2 t}\nabla_x u(0,t) \dd t,
\end{align*}
for any $\lambda>0$.  Since the Laplace transform is injective, this completes the proof.
\end{proof}


 By the same argument, we have the following result for the Cauchy problem.
 
 \begin{lemma}[Relationship between fractional heat and wave equation for the Cauchy problem]
	Let $u$ be the  solution of \eqref{eq:fhe} and $w$ be the solution of \eqref{eq:fwCauchy} for $u_0 \in C^\infty_0(\mathbb R^N)$ with $\supp(u_0) \subset B_L(0)$ for some $L >0$. Then the following holds: 
	\begin{enumerate}
		\item $w(0,t) = 0$ for any $t>0$ if and only if $u(0,t) = 0$ for any $t>0$; 
		\item $\nabla_x w(0,t) = 0$ for any $t>0$ if and only if $\nabla_x u(0,t) = 0$ for any $t>0$.
	\end{enumerate}
\end{lemma}

We remark that, to make the argument of the above Lemma rigorous, we need to prove the  regularity of solutions of the elliptic problems \eqref{eq:lap2}--\eqref{eq:lap3} ($C^0$ or $C^1$ regularity, respectively); to this end, we apply the bootstrap argument of \cite{MR3331523}, which requires us only to prove the boundedness of solutions.

\begin{lemma}[$L^\infty$-bound]
Let $\Omega$ be a bounded $C^{1,1}$ domain in $\R^N$ and $u_0 \in C^\infty_0(\R^N)$ with $\supp u_0 \subset B_L(0) \Subset \Omega$. 
For each $\lambda >0$, let  us consider 
	\begin{align}\label{eq:lap2-2}
	\begin{cases}
	(-\Delta)^s W_\lambda(x) + \lambda^2 W_\lambda(x) = u_0(x), & x \in \Omega, \\
	W_\lambda(x) = 0, & x \in \R^N \setminus \Omega.
	\end{cases}
	\end{align}
Then  $W_\lambda \in L^\infty(\R^N)$ and there exists a constant $C>0$, independent of $\lambda>0$, such that $|W_\lambda| \le C$ in $\R^N$. 
\end{lemma}

\begin{proof}
	Let us consider a ball such that $\bar \Omega \subset B_R(0)$ and the problem 
	\begin{align*}
	\begin{cases}
	(-\Delta)^s v(x) = 1, & x \in B_R(0), \\
	v(x) = 0, & x \in \R^N \setminus B_R(0).
	\end{cases}
	\end{align*}
	Then 
	$$|W_\lambda|(x)\le \|u_0\|_{L^\infty(\R^N)} v(x), \quad x \in \Omega.$$
	Indeed, let us consider the function $V=v\| u_0\|_{L^\infty(\R^N)} \in L^\infty(\R^N)$ 
	and set $f_\pm=W_\lambda \pm V$ to obtain
	\begin{align*}
	\begin{cases}
	(-\Delta)^s f_+ + \lambda^2 f_+ \ge 0 \text{ and } (-\Delta)^s f_- + \lambda^2 f_-\le 0, & x \in \Omega, \\
	f_+ \ge 0 \text{ and } f_-\le 0, & x \in \R^N \setminus \Omega.
	\end{cases}
	\end{align*}
	Hence, by the maximum principle, $$-V(x) \le W_\lambda(x) \le V(x), \quad x \in \Omega.$$
	
\end{proof}

\vspace{5mm}
\section*{Acknowledgments}

N.~De Nitti is a member of the Gruppo Nazionale per l'Analisi Matematica, la Probabilit\`a e le loro Applicazioni (GNAMPA) of the Istituto Nazionale di Alta Matematica (INdAM).  He has been supported by the Alexander von Humboldt Foundation and by the TRR-154 project of the Deutsche Forschungsgemeinschaft (DFG, German Research Foundation).  

S. Sakaguchi has been supported by the Grants-in-Aid for Scientific Research (B) ($\#$18H01126 and $\#$17H02847) of Japan Society for the Promotion of Science.


\vspace{5mm}
\bibliographystyle{abbrv} 
\bibliography{FractionalHeatFlow-ref.bib}
\vfill 
\end{document}